\documentclass{llncs}
\usepackage[T1]{fontenc}
\usepackage[utf8]{inputenc}
\usepackage{authblk}
\usepackage[centertags]{amsmath}
\usepackage{graphicx}
\usepackage{amscd}
\usepackage{amsmath}
\usepackage{color}
\usepackage{listings}
\usepackage{amsfonts}
\usepackage{amssymb}
\usepackage{newlfont}
\usepackage{comment}
\usepackage{appendix}
\usepackage{color}
\usepackage{hyperref}
\usepackage{subfigure}
\usepackage{csquotes}
\usepackage{footnote}
\usepackage{bbm}
\usepackage{todonotes}

\usepackage{geometry}
\newcommand*{\un}[1]{\underline{#1}}
\newcommand*{\be}{\begin{equation}}
\newcommand*{\ee}{\end{equation}}
\newcommand*{\ba}{\begin{aligned}}
\newcommand*{\ea}{\end{aligned}}
\newcommand{\CL}{\mathcal{L}}

\newcommand{\ve}{\varepsilon}

\newtheorem{condition}{Condition}

\begin{document}

\title{Transfinite fractal dimension of trees and hierarchical scale-free graphs}


\author{{J\'ulia Komj\'athy \thanks{j.komjathy@tue.nl}}\inst{1} \and Roland Molontay \thanks{molontay@math.bme.hu}\inst{2} \inst{3}
\and K\'aroly Simon \thanks{simonk@math.bme.hu}\inst{2} \inst{3}}

\institute{Department of Mathematics and Computer Science, Eindhoven University of Technology, Netherlands\\
\and
Department of Stochastics, Budapest University of Technology and Economics, Hungary
\and
MTA-BME Stochastics Research Group, Hungary
}
\authorrunning{J\'ulia Komj\'athy \and Roland Molontay \and K\'aroly Simon}

\maketitle

\begin{abstract}
{In this paper, we introduce a new concept: the  transfinite fractal dimension of graph sequences motivated by the notion of fractality of complex networks proposed by Song \emph{et al.} We show
that the definition of fractality cannot be applied to networks with `tree-like' structure and exponential growth rate of neighborhoods. However, we show that the definition of fractal dimension could
be modified in a way that takes into account the exponential growth, and with the modified definition, the fractal dimension becomes a proper parameter of graph sequences. We find that this parameter is related to
the growth rate of trees. We also generalize the concept of box dimension further and introduce the transfinite Cesaro fractal dimension. Using rigorous proofs we determine the optimal box-covering and
 transfinite fractal dimension of various models: the hierarchical graph sequence model introduced by Komj\'athy and Simon, Song-Havlin-Makse model, spherically symmetric trees, and
supercritical Galton-Watson trees.}

\vspace{4mm}

\end{abstract}

\section{Introduction}
The study of complex networks has received immense attention recently, mainly because networks are used in several disciplines of science, such as in Information Technology (World Wide Web, Internet), Sociology (social relations), Biology (cellular networks) etc. Understanding the structure of such networks has become essential since the structure affects their performance, for example the topology of social networks influences the spread of information and disease. In most cases real networks are too large to describe them explicitly. Hence, models must be considered. A network model can be static, i.e., it models a snapshot of the network, such as \cite{chung2002connected,erdos1960evolution,molloy1995critical} or dynamic, i.e., the model mimics the evolution of the network on the long term \cite{barabasi1999emergence}.

Many networks were claimed to show self-similarity and fractal behaviour \cite{gallos2007review}. Heuristically, fractality of a network means that the network looks similar to itself on different scales: if one zooms in on a sub-network, one is expected to see the same qualitative behaviour as in the whole network. Unfortunately, most of the classical random graph models (e.g. the Chung-Lu model \cite{chung2002connected}, the configuration model \cite{Bollobas01} or the preferential attachment model \cite{barabasi1999emergence}) do not model the phenomenon of hierarchical or self-similar structure in the network. To solve this problem, Barab\'asi, Ravasz and Vicsek introduced  deterministic
hierarchical scale-free graphs constructed  by a method which is common in generating fractals \cite{barabasi2001deterministic}. Their proposed deterministic, hierarchical network (that we call "cherry") can be seen in Figure \ref{fig:cherry},  Ravasz and Barab\'asi improved this original "cherry" model to further accommodate \emph{clustering}, that is, the presence of local triangles; and obtained similar
clustering behavior to many real-world  networks \cite{ravasz2003hierarchical}. 

 A similar fractal based approach was introduced by Andrade \emph{et al.} \cite{andrade2005apollonian}, the Apollonian networks. The name comes from the generating method of the model, that uses Apollonian circle packings to obtain the network. Apollonian networks were generalized to higher dimensions and 
investigated by Zhang \emph{et al.} \cite{zhang2006high,zhang2006evolving}. For further fractal related network
models see e.g.\ \cite{dorogovtsev2002pseudofractal,karci2015new,zhang2006deterministic}. The first and last authors of the present paper generalized the "cherry" model of
 \cite{barabasi2001deterministic} by introducing a general hierarchical graph sequence derived from a graph directed self-similar fractal \cite{komjathy2011generating}. 
 We mention that there are also some natural, namely, spatial, random network models where a hidden hierarchical srtructure is embedded in the graph: Heydenreich, Hulshof and Jorritsma showed the existence of a hierarchical structure in the scale-free percolation model~\cite{heydenreich2017structures}.

To accommodate the observed fractality in network models is one side of the coin. The other side is to identify fractality and the presence of self-similarity of complex networks beyond the heuristics.  A method was proposed by Song, Havlin and Makse  \cite{song2005self}; they suggested that the procedure for networks must be similar to that of regular fractal objects: using the box-covering method. Once a network is covered with boxes, the notion of fractality stands for a polynomial relation between the number
of boxes needed to cover the network and the size of the boxes. The polynomial relation was verified in many real-world networks, e.g. the World Wide Web, actor collaboration network
and protein interaction networks \cite{kim2007fractality,rozenfeld2009fractal,song2006origins}. The (approximate) exponent of this relation gives, heuristically, the box-covering dimension of the network. The fractality and self-similarity of complex networks was
investigated in several further articles 
\cite{gallos2007review,rozenfeld2007fractal,song2007calculate,song2006origins}
and we give a short review of this topic in Section \ref{fractality}. 

While many real life-networks do satisfy an approximate polynomial relationship between box sizes and the number of boxes needed,  for example, the Internet at router level or most of the social networks \cite{gallos2007review,nagy2018data} do not. In these and many other cases, at least locally, the neighborhood of a vertex grows exponentially as the radius grows. In these cases, no polynomial relationship can be found. 
On the other hand, for network models with non-polynomial local growth rate a new definition of box dimension is needed, that is the \emph{transfinite fractal dimension} developed by Rozenfeld \emph{et al.} in \cite{havlin2007fractal,rozenfeld2009fractal} (see \eqref{eq:exp-rel} below). As the main point of this article, we make the heuristic definition mathematically rigorous.

To obtain a mathematically rigorous yet natural definition, we consider the dimension of graph sequences. This is \emph{natural} for two reasons. The first reason is that for finite networks, once the box size exceeds the diameter of the network, a single box is enough to cover the whole network, and any relation between the sizes of the boxes and their number can only be valid in a given range of box sizes, hence, no true `dimension' concept can exist in a mathematical sense that resembles box-covering. The second reason is that  many networks grow in size as time passes, hence, it is natural to consider sequences of graphs with more and more vertices.
  
 We test our definition of transfinite fractal dimension on some of the above mentioned models that intuitively contain hierarchical structures. Namely, we test the definition on the above mentioned  "cherry" model by Barab\'asi \emph{et al.} \cite{barabasi2001deterministic} and its generalisation, the hierarchical graph sequence proposed by Komj\'athy and Simon \cite{komjathy2011generating}, and a recursively defined hierarchical model, proposed by Song, Havlin and Makse \cite{song2006origins}. We further test our definition on random and deterministic trees: branching processes and spherically symmetric trees. Recursively defined trees naturally contain hierarchy; namely, a subtree of a vertex may resemble the whole tree. While there is no obvious direct relationship between the optimal number of boxes to cover a network and the exponential growth rate of neighborhood sizes, on the studied models we confirm that the two parameters are indeed intimately related. Our definition of transfinite fractal dimension gives a natural parameter that indeed captures the exponential growth of the neighborhoods of the model in a quantitative way. On trees, we show that the box-covering is indeed related to the (exponential) \emph{growth rate}; introduced by Lyons and Peres \cite{lyons2016probability}.

In the literature, box-covering is determined mostly by approximation algorithms \cite{deng2016performance,song2007calculate}, while our method is
rigorous on the above mentioned models. It is an interesting further direction of  research to see how well approximation algorithms perform on the models that we rigorously study in this paper. We mention that due to the exponential scaling; our definition is \emph{robust} in the sense that if an approximation algorithm is able to approximate the optimal number of boxes of a network up to finite constant factors, than the empirical box dimension will confirm the theoretical value that we derive here.

 We mention that other graph dimension concepts have been also generalized to the infinite case such as the metric and partition dimensions \cite{caceres2012metric,tomescu2009metric}. In \cite{baccelli2018dimension}  the  Minkowski and Hausdorff dimensions are
defined for unimodular random discrete metric spaces while \cite{baccelli2018dimension2} sheds light on the connections between these notions and the polynomial
growth rate of the underlying space. In this work, we focus on the generalization of the box-covering dimension. For a recent survey about other notions of dimension we refer to \cite{rosenberg2018survey}.

\emph{Structure of  the paper.} After a short review of the topic of network fractality
by Song \emph{et al.} \cite{song2005self}, in Section \ref{fractality} below we introduce the definition of box dimension for graph sequences, the \emph{transfinite fractal dimension} and a generalized version, the \emph{transfinite Cesaro fractal dimension}.  In Section \ref{model} we determine the optimal number of boxes
needed to cover the hierarchical graph sequence model \cite{komjathy2011generating}.  We find that the hierarchical graph sequence model \cite{komjathy2011generating} does not have a finite box dimension (based on the usual definition assuming polynomial growth) but the transfinite dimension exists (based on our new definition assuming exponential growth). In Section \ref{sec:shm} we investigate the optimal boxing  and transfinite  dimension of a fractal network model introduced by Song, Havlin and Makse \cite{song2006origins}.
In Section \ref{boxingtrees} we determine the optimal boxing and the  transfinite dimension of some deterministic and random trees, in particular, spherically symmetric trees and Galton-Watson branching processes, and
relate the obtained dimension to the growth rate of trees introduced by Lyons and Peres \cite{lyons2016probability}. Section \ref{conc} concludes the work.

\section{Fractal scaling in complex networks and concepts of box dimension}\label{fractality}

In this section, we review the concepts of box dimension of networks proposed by Song \emph{et al.} in \cite{song2005self} and  the transfinite dimension proposed by Rozenfeld \emph{et al.} \cite{havlin2007fractal,rozenfeld2009fractal},  and make the two concepts rigorous by giving mathematically precise definitions. These yield Definitions \ref{boxdim} and \ref{modboxdim} of box dimension and transfinite fractal dimension, respectively. The technique Song \emph{et al.} in \cite{song2005self} proposed for identifying the presence of fractality in complex networks is analogous to
that of regular fractals. Namely, for `conventional' fractal objects in the Euclidean space (e.g. the attractors of iterated function systems), a basic tool is the box-covering method \cite{falconer2004fractal}. This method works as follows: one covers the fractal set by smaller and smaller sizes of boxes, and finds the polynomial relationship between the optimal number of boxes used versus the side-length of the boxes; as the side-length goes to zero. A similar method can be applied to networks that we describe now.
Since the Euclidean metric is not relevant for graphs, it is reasonable to use a natural metric, namely the shortest path length between two vertices. In the case of unweighted graphs this metric is called the \emph{graph distance metric}.

The method works as follows \cite{song2007calculate}: For a given network $G$ with $N$ vertices, we partition the vertices into subgraphs (boxes) with diameter at most $\ell-1$ (it is illustrated in
Figure \ref{fig:renorm}). The minimum number of boxes needed to cover the entire network $G$ is denoted by $N_B(\ell)$. Determining $N_B(\ell)$ for any given $\ell\ge 2$ belongs to a family of NP-hard
problems but in practice various algorithms are adopted to obtain an approximate solution \cite{song2007calculate}. In accordance with regular fractals, Song \emph{et al.} proposed to define the fractal dimension
or box dimension $d_B$ of a finite graph by the approximate relationship:
\be \label{powerbox}
N_B(\ell)/N \approx: \ell^{-d_B},
\ee
i.e., the required number of boxes scales as a power of the box size, and the dimension is the absolute value of the exponent. In their reasoning, the relationship in \eqref{powerbox} should hold for a wide range of values $\ell$ with the same exponent $d_B$.

\begin{figure}[ht]
	\centering
		\includegraphics[width=100mm]{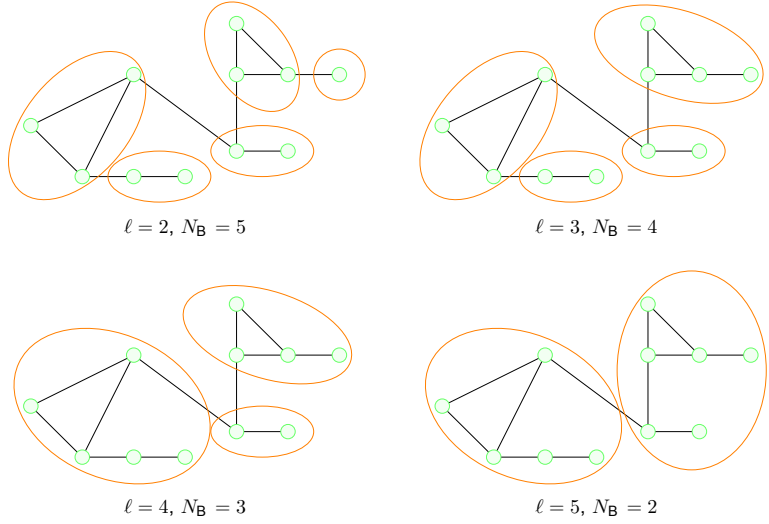}
	\caption{The box-covering algorithm as employed in a network demo of eleven nodes for different box sizes $\ell$. The figure was adapted by the author from \cite{song2005self}.}
	\label{fig:renorm}
\end{figure}

According to this method, the power form of \eqref{powerbox} (with a finite $d_B$) can be verified by plotting and fitting in a number of real-world networks such as WWW, actor collaboration network
and protein interaction networks \cite{song2006origins}. For these networks, a finite box-dimension exists. However, a large class of complex networks  (called non-fractal networks) is characterised by a sharp decay of $N_B$ with increasing $\ell$, i.e., has
infinite fractal dimension, for example, the Internet at router level or most of the social networks \cite{gallos2007review,nagy2018data} falls into this category. To distinguish these cases, they introduced the \emph{concept of fractality} as follows \cite{song2005self}:

The \emph{fractality} of a finite network (also called fractal scaling or topological fractality) means that there exists a power relation between the minimum number of boxes needed to cover the entire network and the size of the boxes.

In other words, as mentioned above, equation \eqref{powerbox} must hold for a $d_B$ for a wide range of $\ell$ for a network to show fractality.
Although it is possible to ascertain the fractal dimension with this description and \eqref{powerbox} using approximation methods, here we develop a rigorous mathematical definition shortly below. The need for a rigorous definition arises naturally: first, the relation~\eqref{powerbox} is approximate, and second, it is hard to quantify what may one call a wide range of $\ell$.

 To motivate our choice of definition, when considering regular fractal objects (that is, sets embedded in $\mathbb R^d$ for some integer $d$) the box dimension\footnote{Also called Minkowski-dimension.} is defined as the \emph{limit} of the reciprocal of the ratio of the logarithm of the number of boxes and the logarithm of the box
size, as the box size \emph{tends to $0$}. This definition would make no sense with respect to networks, since the graph distance can not be less than $1$. On the other hand, tending to \emph{infinity} with the box size might be a solution
if the network itself grows, or is infinite to start with. For this reason, we should consider \emph{graph sequences}. Several real-world networks (collaboration networks, WWW) grow in size as time proceeds, therefore it is reasonable to consider graphs of growing size, denoted by $\left\{ G_n\right\}_{n\in \mathbb N}$ (where $\mathbb N$ stands for the set of natural numbers). For infinite networks such as $\mathbb Z^d$, one can choose a root vertex (e.g. the origin) as a point of reference and consider subgraphs of the underlying infinite network centered around the reference vertex that exhaust the infinite graph (e.g. $G_n:=[-n,n]^d$ for $\mathbb Z^d$).

To be able to define the box dimension of a graph sequence, we define the above mentioned boxes of size $\ell$ first.

\begin{definition}[$\ell$-box]\label{lbox} Consider two vertices $u, v$ in a graph $G$. Let $\Gamma(u,v)$ denote the set of all paths connecting $u,v$ within $G$. The length of a path $\pi$ is defined as the number of edges on $\pi$ and is denoted by $|\pi|$. The graph distance between two vertices $u,v$ in a graph $G$ is defined as $d_G(u,v)=\min\{|\pi|: \pi\in \Gamma(u,v)\}$.
We say that a subgraph $H$ of a graph $G$ is an $\ell$-box if $d_H(u,v)\le \ell-1$ holds for all $u,v\in H$.
 
\end{definition}
Our first definition is the rigorous form of \eqref{powerbox}:
\begin{definition}[Box dimension]\label{boxdim} The box dimension $d_B$ of a graph sequence $\left\{ G_n\right\}_{n\in \mathbb N}$ is defined as
\be
d_B\left( \left\{ G_n\right\}_{n\in \mathbb N} \right):=\lim_{\ell \to \infty} \lim_{n \to \infty} \frac{\log \left( N_B^n(\ell)/|G_n|\right)}{- \log \ell},
\ee
if the limit exists; where $N_B^n(\ell)$ denotes the minimum number of $\ell$-boxes needed to cover $G_n$, and $|G_n|$ denotes the  number of vertices in $G_n$.
\end{definition}

 Note that this definition indeed gives back \eqref{powerbox}, since it means that, for each $\ve>0$, there exists  $\ell(\ve), n(\ve, \ell)$ such that whenever $\ell\ge \ell(\ve)$, \emph{every} $G_n$ with $n\ge n(\ve, \ell)$ can be convered with $|G_n|\ell^{-d_B\pm \ve}$ many $\ell$-boxes. 
We comment on the order of the limits in the previous definition. It is natural question to ask whether the limiting operations can be interchanged. Considering the fact that the number of boxes needed
to cover $G_n$ is $N_B^n(\ell)=1$ if $\ell > \mathrm{diam}(G_n)$, it is meaningless to change the order of the limits.

It is not hard to see that this definition of fractality cannot be applied to networks with exponential growth rate of neighborhoods. Indeed, in this case the optimal number of boxes
does not scale as a power of the box size. On the other hand, the box-covering method yields  another  natural parameter if we modify the required functional relationship
between the minimal number of boxes and the box size as in the transfinite fractal cluster dimension by Rozenfeld \emph{et al.} \cite{havlin2007fractal,rozenfeld2009fractal}). Namely, we might consider finding $\tau$ that satisfies
\be\label{eq:exp-rel}
	N_B(\ell)/N \approx: e^{-\tau \cdot \ell}
	\ee
for a wide range of $\ell$.
Again, we make this concept rigorous and quantifyable by defining the \emph{transfinite fractal} dimension of graph sequences similarly:
\begin{definition}[Transfinite fractal dimension]\label{modboxdim} The transfinite fractal dimension $\tau$ of a graph sequence $\left\{ G_n\right\}_{n\in \mathbb N}$ is defined by
\be \label{tau_lim}
\tau \left( \left\{ G_n\right\}_{n\in \mathbb N} \right):= \lim_{\ell \to \infty} \lim_{n \to \infty} \frac{\log \left(N_B^n(\ell) / |G_n| \right)}{-\ell},
\ee
if the limit exists; where $N_B^n(\ell)$ denotes the minimum number of $\ell$-boxes needed to cover $G_n$, and $|G_n|$ denotes the  number of vertices in $G_n$.
\end{definition}

\begin{remark} We call $\tau$ the transfinite fractal dimension or `growth-constant' since it captures how spread-out neighborhoods of vertices are, on an exponential scale.
\end{remark}

We shall see in Section \ref{spherical} that for some models with exponentially growing neighborhood sizes the limit in \eqref{tau_lim} does not exist but the limit of the Cesaro means does. This yields 
the \emph{transfinite Cesaro fractal box dimension}. We modify Def.~\ref{modboxdim} by considering the Cesaro-sum instead of the pure limit in $n$:

\begin{definition}[Transfinite Cesaro fractal dimension]
\label{def::ourdim_def2}
The transfinite Cesaro fractal dimension $\tau^*$ of a graph sequence $\left\{ G_n\right\}_{n\in \mathbb N}$ is defined by
\begin{equation}
\tau^*\left( \left\{ G_n\right\}_{n\in \mathbb N} \right):=\lim_{\ell \to \infty} \lim_{n \to \infty} \frac1n \sum_{i=1}^n \frac{\log \left( N_B^{i+\ell}(\ell) / {|G_{i+\ell}|}\right)}{-\ell},
\end{equation}
if the limit exists; where $N_B^n(\ell)$ denotes the minimum number of $\ell$-boxes  needed to cover $G_n$, and $|G_n|$ denotes the  number of vertices in $G_n$.
\end{definition}
The definition of box dimension for graph sequences with exponentially growing neighborhood sizes was first introduced in the Bachelor thesis of the second author \cite{molontay2013networks}, that is an unpublished work. Dai \emph{et al.} \cite{dai2017modified} studied the transfinite fractal dimension of the weighted version of the model in \cite{komjathy2011generating} and a similar weighted fractal network \cite{dai2015modified}.
In what follows we investigate graph sequences with exponentially growing neighborhood sizes, and determine their transfinite fractal as well as transfinite Cesaro fractal dimension. These examples shall demonstrate that our definition is a natural one.

	\section{Optimal boxing of a hierarchical scale-free network model based on fractals}
	\label{model}
	\subsection{Description of the model}
	\label{modeldef}
	This model was introduced by the first and last author of this article. In this section, we follow the notation of  \cite{komjathy2011generating}.
We start with an arbitrary initial bipartite graph $G$, the \emph{base graph}, on $N$ vertices and we define a hierarchical sequence of deterministic graphs $\left\{ \mathrm{HM}_n \right\}_{n\in \mathbb N}$ in a recursive manner.
Let $V(\mathrm{HM}_n)$, the set of vertices of $\mathrm{HM}_n$ be $\left\{0,1,\dots ,N-1\right\}^n$. The construction of $\mathrm{HM}_n$ from $\mathrm{HM}_{n-1}$ works by taking $N$ identical copies of $\mathrm{HM}_{n-1}$, corresponding to the $N$ vertices of
the base graph $G$. Next, we construct the edges between the copies described in Def.~\ref{edgeset_def} below. Along these lines, $\mathrm{HM}_n$ contains $N^{n-1}$ copies of $\mathrm{HM}_1$, connected in a hierarchical way.

Let $G$, our \emph{base graph}, be any labeled \emph{bipartite graph} on the vertex set $\Sigma=\Sigma_1=\left\{0,\dots ,N-1\right\}$ with bipartition $\Sigma=V_1 \cup V_2$, such that one of the end points of any edge in  $G$ is in $V_1$, while the other one is in $V_2$. We write $n_i:=|V_i|$, $i=1,2$ and $E(G)$ for the edge set of $G$. We denote edges as ${ x \choose y }$. The vertex set of $\mathrm{HM}_n$ is then given by  $\Sigma_n= \{ (x_1 x_2\dots x_n): x_i \in \Sigma\}$, all words of length $n$ above the alphabet  $\Sigma$. In order to define the edge set of $\mathrm{HM}_n$, we need to introduce some further definitions \cite{komjathy2011generating}.
\begin{figure}[t]
	\centering
		\includegraphics[width=120mm]{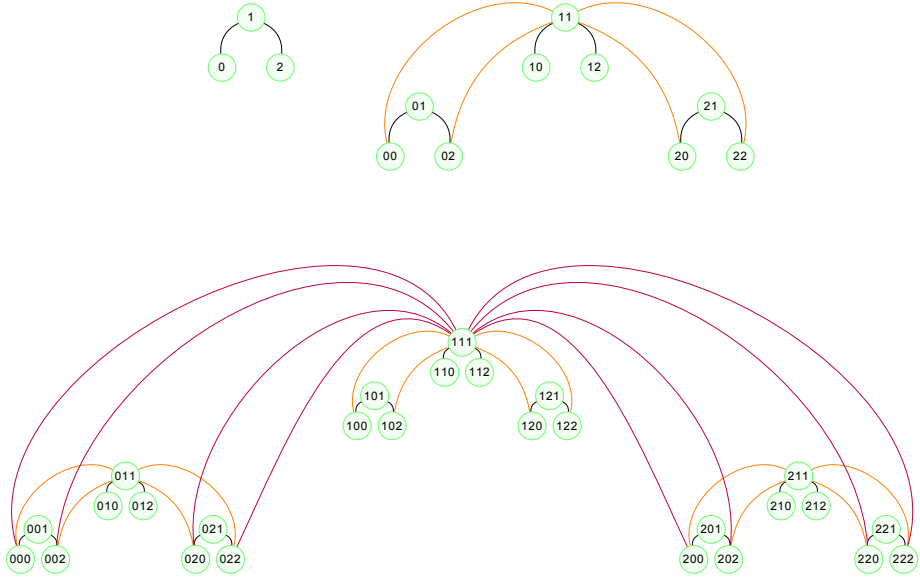}
	\caption{The first three elements of the \enquote{cherry} model: $\mathrm{HM}_1$,  $\mathrm{HM}_2$ and $\mathrm{HM}_3$. The figure was adapted by the authors from \cite{komjathy2011generating}.}
	\label{fig:cherry}
\end{figure}

\begin{definition}\label{edgeset_def} \
\begin{enumerate}
  \item We assign a type to each element of $\Sigma$. Namely,
$
\text{typ}(x):=
\left\{
  \begin{array}{ll}
    1, & \hbox{if $x \in V_1$;} \\
    2, & \hbox{if $x \in V_2$.}
  \end{array}
\right.
$
  \item For $i=1,2$, we say that the {\bf type} of a word $\un z=(z_1z_2\dots z_n )\in \Sigma _n$ equals $i$ and write $\text{typ}(\un z)=i$, if $\text{typ}(z_j)=i$, for all $ j=1,\dots, n$.  Otherwise $\text{typ}(\un z):=0$.

\item  For  $\un x= (x_1\dots x_n), \un y=(y_1\dots y_n) \in \Sigma_n$ we denote the {\bf common prefix} by
\[  \un x \wedge \un y := (z_1 \dots z_k) \text{ s.t. } x_i=y_i=z_i,  \forall i=1,\dots, k \text{ and } x_{k+1}\neq y_{k+1}, \]
\item and the {\bf postfixes} $\tilde {\un x}, \tilde {\un y} \in \Sigma_{n-|\un x \wedge \un y|}$ are determined by
    \[\un x =: (\un x \wedge \un y) \tilde {\un x},\  \un y =: (\un x \wedge \un y) \un{\tilde y}, \] where the concatenation of the words $\un a,\un b$ is denoted by $\un a \un b$.
\end{enumerate}\end{definition}
Next, we define the edge set $E(\mathrm{HM}_n)$. Two vertices $\un x$ and $\un y$ in $\mathrm{HM}_n$ are connected by an edge if and only if the following criteria hold:
\begin{description}
  \item[(a)]  One of the postfixes $\un {\tilde x}, \un {\tilde y}$ is of type $1$, the other is of type $2$,
  \item[(b)]  for each $i>|x\wedge y|$, the coordinate pair ${ x_i \choose y_i}$ forms an edge in $G$.
\end{description}
\begin{remark}[Hierarchical structure of $\mathrm{HM}_n$]\label{hierarch} For every initial digit $x \in \{0,1, \dots, N-1\}$, consider the set $W_x$ of vertices $(x_1 \dots x_n)$ of $\mathrm{HM}_n$ with $x_1=x$. Then the induced subgraph on $W_x$ is identical to $\mathrm{HM}_{n-1}$.
\end{remark}
The following two examples satisfy the requirements of our general model.
\begin{example}[Cherry]\label{cherry}
The \enquote{cherry} model was introduced in \cite{barabasi2001deterministic}, and is presented in Figure \ref{fig:cherry}: Let $V_1=\{1\}$ and $V_2=\{0,2 \}$, $E(G)=\left\{ (1,0), (1,2)\right\}$.
\end{example}
\begin{example}[Fan]\label{gkal} Our second example is called \enquote{fan}, and is defined in Figure \ref{fig:fan}. Note that here $|V_1|>1$.
\end{example}
\begin{figure}[htbp]
	\centering
	\includegraphics[width=120mm]{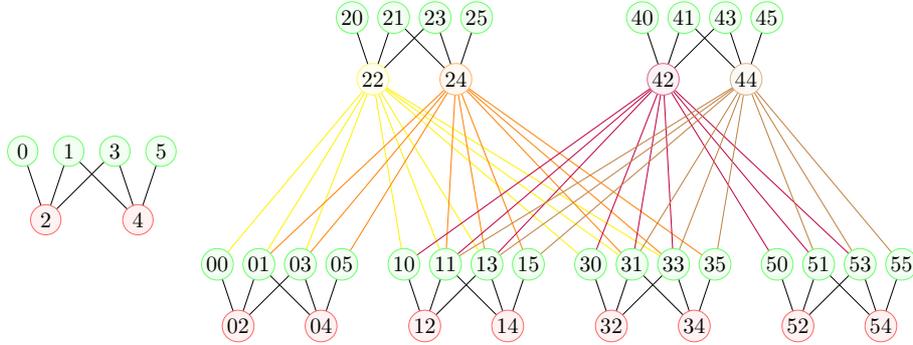}
	\caption{The first two elements of the \enquote{fan}. Here $V_1=\{2,4\}$ and $V_2$=\{0,1,3,5\}. (They contain additionally all loops.) The figure was adapted by the authors from \cite{komjathy2011generating}.}
	\label{fig:fan}
\end{figure}

\subsection{The optimal box-covering}
	\label{optimal}
	
	In this section, we determine the optimal box-covering of the hierarchical graph sequence model introduced before. We find that the optimal number of boxes does not scale as a power of the box size,
meaning that this graph sequence has no finite box dimension, on the other hand, the transfinite fractal dimension exists and is a meaningful parameter.

\begin{theorem}
\label{hier_tau}
The hierarchical graph sequence $\left\{ \mathrm{HM}_n\right\}_{n\in \mathbb N}$ is not fractal, but transfractal. That is, its fractal dimension (as in Def.~\ref{boxdim}) does not exists, while its transfinite fractal dimension (as in Def.~\ref{modboxdim}) exists and equals \be \label{tau_value}
\tau \left( \left\{ \mathrm{HM}_n\right\}_{n\in \mathbb N} \right) = (\log N)/2,
\ee
where $N$ denotes the number of vertices in the base graph $G$ of $\left\{ \mathrm{HM}_n\right\}_{n\in \mathbb N}$.
\end{theorem}

In the rest of this section we investigate the optimal boxing of the model for certain box sizes, namely those that can be expressed as $\mathrm{diam}(\mathrm{HM}_k)+1$. We thus define
\be\label{ellk-bkn} \ell_k:=\mathrm{diam}(\mathrm{HM}_k)+1. \ee  
Using this notation, we prove Theorem \ref{hier_tau}. The analysis of the box-covering consists of two main parts: giving upper and lower bound on $N_B^n(\ell_k)$.

\begin{figure}[t]
	\centering
		\includegraphics[width=120mm]{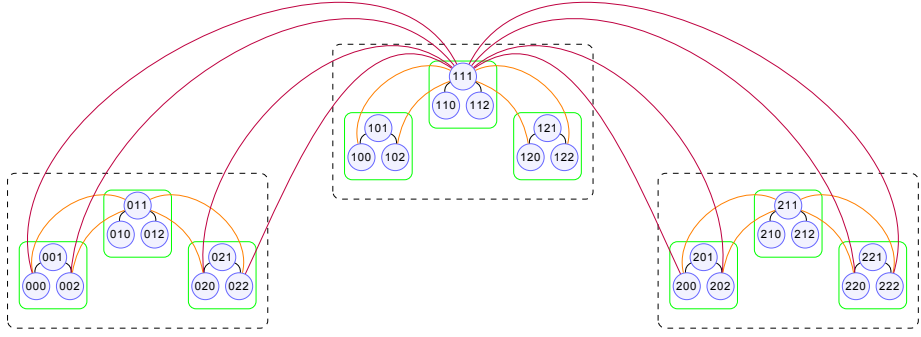}
	\caption{The third iteration of an instance of the hierarchical graph sequence model, called "cherry" model:  $\mathrm{HM}_3$. The boxing of the graph is also highlighted: the green boxes illustrate an optimal 3-boxing and the dashed boxes show an optimal 7-boxing of the graph, i.e. $N_B(3)=9$ and $N_B(7)=3$. The transfinite dimension of the model is $\tau(\left( \mathrm{HM}_n\right)_{n\in \mathbb N})=(\log K)/2$, here the base graph is on $K=3$ vertices.}
	\label{fig:boxcherry}
\end{figure}

\subsubsection{Upper bound on the optimal number of boxes.}

The following lemma is a useful tool to examine the box dimension of the graph sequence. Here we use the notation of Section \ref{model}.

\begin{lemma}
\label{diam_lem}
The diameter of the hierarchical graph sequence model $\mathrm{HM}_n$ (defined in Section \ref{modeldef}) is $\mathrm{diam}(\mathrm{HM}_n)=2(n-1)+\mathrm{diam}(G)$.
\end{lemma}
The proof can be found in the Appendix. Its heuristics is as follows: between any two vertices with names $\underline x=(x_1\dots x_n), \underline y=(y_1\dots y_n)$ one can construct a path by gradually changing the coordinates of the names starting from the end of the name. In total, one needs to change all the coordinates of $\underline x$ and $\underline y$ at most once (using $2(n-1)$ edges) in order to reach the same copy of the base graph $G$. In this copy, one needs to take at most $\mathrm{diam}(G)$ steps to connect the two paths.

Recall $\ell_k$ from \eqref{ellk-bkn}. The following lemma gives an upper bound on $N_B^n(\ell_k)$, the number of boxes needed to cover $\mathrm{HM}_n$ with boxes of diameter at most $\ell_k$. 

\begin{lemma}[Upper bound on the number of boxes]
\label{lemma_upper} For all $k\ge n$, $N_B^n(\ell_k)=1$, while for all $n>k$,
\be \label{ubl} N_B^n(\ell_k) \leq N^{n-k}
\ee
\end{lemma}
\begin{proof}
 Recall that by construction, $\mathrm{HM}_n$ consists of $N^{n-k}$ copies of $\mathrm{HM}_k$. Indeed, each vertex in $\mathrm{HM}_n$ has a code of length $n$, where each letter in the code is in $\{0,\dots, N-1\}$. Let us define the $\ell_k$-boxes as follows: every vertex, starting with the same word of length $n-k$, constitutes to one box. This box is a copy of $\mathrm{HM}_{k}$ by the definition of the model. There are $N^{n-k}$ possible ways to start an $n$-length code, hence the number of boxes is $N^{n-k}$.
The diameter of each box is then $\mathrm{diam}(\mathrm{HM_k})=\ell_k-1$ per definition, hence, these are proper $\ell_k$-boxes.
\end{proof}

 We continue giving lower bounds. Note that lower bounds are not that easy, since the `long' edges connecting different copies of $\mathrm{HM_k}$ within $\mathrm{HM_n}$ might allow for a better boxing than using the directly observable hierarchical structure, see Figures  \ref{fig:fan} and \ref{fig:boxcherry}. First we investigate the case $k=1$, i.e., $\ell=\ell_1=\mathrm{diam}(G)$.

\begin{lemma}[Lower bound on $ N_B^n(\ell_1)$] \label{lemma_lower}
For all $n \geq n_1+1,$
\be
N_B^n(\ell_1) \geq N^{n-n_1},
\ee where $n_q:=|V_q|$, $q\in\{1,2\}$ and we assume that $n_1 \leq n_2$ without loss of generality.
\end{lemma}
\begin{proof}

We start observing that $\mathrm{diam}(G) \leq 2n_1$ since we assumed that G is bipartite and connected.
It is enough to show that we can find $N^{n-n_1}$ \emph{witness} vertices in $\mathrm{HM}_n$ for all $n \geq n_1+1$, such that the pairwise distances between these witnesses are greater than $2n_1$ (hence greater than $\mathrm{diam}(G))$ so they all must be in distinct $\ell_1$-boxes\footnote{By the definition of diameter, in any given copy of $\mathrm{HM_1}$ there are two vertices that are at distance $\mathrm{diam}(G)$ from each other, but it is unclear that once having many copies of $\mathrm{HM}_1$, how far are vertices in different copies  of $\mathrm{HM_1}$ from each other, allowing for a possibly better boxing.}.

First we investigate the case when $n=n_1+1$. In this case we need $N^{n-n_1}=N$ witnesses. For each base letter $\{0,1,\dots, N-1\}:=[N]$ we construct one witness vertex. 
Recall from Def.~\ref{edgeset_def} that the type of a letter $x\in [N]$ is $i\in\{1,2\}$ if the vertex $x\in G$ is in partition $V_i, i\in\{1,2\}$. 
We say that a vertex $\un z=\un z_x$ is a witness for $x\in [N]$ if its code starts with $x$ and the consecutive letters keep alternating the type, i.e., in case $x$ was type $1$ than the next letter is type $2$, then again type $1$ and so on. 
Formally, let us find a $\underline z_x=(z_1, \dots, z_n)$ a witness for $x$ that has $z_1=x$ and  typ($z_j$) $\neq$ typ($z_j+1$) for all $j\le n$.
Let us pick an arbitrary $\un z_{x}$ witness for every $x\in [N]$. 

We explain that this collection of vertices is a good witness set, i.e., the distance between any two of them is at least $\ell_1+1$.
To see this, consider $\un z_x$ and $\un z_y $ for $x\neq y$ and note that the codes have no prefix in common ($|\un{z}_{x} \wedge \un{z}_{y}|=0$) and alternating types later on. By point {\bf(a)} after Def.~\ref{edgeset_def}, an edge is between two codes if they start with some common prefix and their postfixes have a type that is different for the two ends of the edge. Since the types of the letters in $\un z_x, \un z_y$ are alternating, any path that tries to connect them needs to change the postfixes $2n_1$ times, once for each length, starting from $\un z_x$ and once for each length starting from the code $\un z_y$.
This means in total at least $2n_1$ in-between vertices, that is, i.e., $2n_1+1$ edges\footnote{More formally, one can apply the construction of the shortest path between any two vertices, explained in the Appendix in the Proof of Lemma \ref{diam_lem}, here, $q=r=i=n_1+1$ in the notation of the proof of Lemma \ref{diam_lem}, thus we need at least $r-1+q-1+1=2n_1+1$ steps on any path between $\un{z}_{x}$ and $\un{z}_{y}$.}. Hence, the distance between $\un z_x, \un z_y$ is at least $2n_1+1$.
Using that $\mathrm{diam}(G) \leq 2n_1$ these witnesses must be in distinct $\ell_1$-boxes, so we need at least $N$ $\ell_1$-boxes to cover $\mathrm{HM}_{n_1+1}$. This proves the lemma for $n=n_1+1$.

Next, we extend this procedure for arbitrary $n \geq n_1+1$. Let $n=n_1+1+j$ for some $j>0$. Recall the hierarchical structure, i.e., the fact that $\mathrm{HM}_{n_1+1+j}$ consists of $N^j$ copies of $\mathrm{HM}_{n_1+1}$. Note also that $\Sigma_j=[N]^j$, all possible words of length $j$.
In $\mathrm{HM}_{n_1+1+j}$ the corresponding witnesses can be chosen as follows: For every $\un{v}\in \Sigma_j$, we define $N$ witnesses that are the concatenation of $\un v$ with the witnesses above, i.e., $\un{v}\un{z}_x\in \Sigma_{n_1+1+j}$ for all $x \in \Sigma$. In words, this means that we find our original $N$ witnesses $(\un z_x)_{x\in[N]}$ in every copy of $\mathrm{HM}_{n_1+1}$ that is embedded within $\mathrm{HM}_{n_1+1+j}$. This way we created $N\cdot N^j=N^{n-n_1}$ witnesses. Thus, once we confirm that their pairwise distance is at least $\ell_1=\mathrm{diam}(G)$, the proof is finished by noting that all of them must be in separate boxes and hence  $N_B^n(\ell_1) \geq N^{n-n_1}$.

To investigate the pairwise distance between the witnesses, we distinguish two cases:  either two witnesses are in the same copy of $\mathrm{HM}_{n_1+1}$, or not. In the first case, the code of the two witnesses is of the form $\un v\un z_x$ and $\un v \un z_y$ for some $\un v \in \Sigma_j, x,y\in \Sigma$.
We have shown in the previous paragraph that the distance between $\un z_x$ and $\un z_y$ is at least $\ell_1$ for all $x\neq y$, i.e., the witnesses within the \emph{same} copy of  $\mathrm{HM}_{n_1+1}$ must be in separate $\ell_1$-boxes.
The pairwise distance between any witnesses $\un v \un z_x$ and $\un v' \un z_{x'}$ for $\un v \neq \un v'$ is also at least $\mathrm{diam}(G)$, by the same argument as the proof of Lemma \ref{lemma_lower}: any path trying to connect them needs to change the types of the postfixes at least $2n_1$ times, yielding at least $2n_1$ in-between vertices and $2n_1+1$ edges.
\end{proof}

The next lemma extends Lemma \ref{lemma_lower} for $\ell_k$.
\begin{lemma}[Lower bound on $N_B^n(\ell_k)$]\label{lemma_lower2}
Using the notation of the previous lemma, the following inequality holds if $n-k \geq n_1:$  
\be\label{eq:lower-nk} N_B^n(\ell_k) \geq N^{n-k+1-n_1}=N^{n-k}\cdot C,\ee
where $C=N^{1-n_1}$ is a fixed constant determined by the base graph $G$.
\end{lemma}
\begin{proof} We start by switching variables. Let $i:=n-k+1$.
Apply Lemma \ref{lemma_lower} with this $i$, to see that $N_B^i(\ell_1)\geq N^{i-n_1}$ for all $i \geq n_1+1$. There, we created $N^{i-n_1}$ vertices in $\mathrm{HM}_i$ with pairwise distance greater than $\mathrm{diam}(G)$. It is enough to show that we can find the same number of witnesses (i.e.,  $N^{i-n_1}$ many) in $\mathrm{HM}_{k+i-1}=\mathrm{HM}_n$ for all $k \geq 1$, such that the pairwise distances between them is at least (by Lemma \ref{diam_lem})
\be\label{eq:ellk} \ell_k=\mathrm{diam}(\mathrm{HM}_k)+1=2(k-1)+\mathrm{diam}(G)+1,\ee this implies \eqref{eq:lower-nk}. 
Recall also that $\mathrm{diam}(G)\le 2 n_1$. Hence it is enough to show that the pairwise distance is at least $2(k-1)+2n_1+1$. 

Now we create the $N^{i-n_1}$ many witnesses. For every witness $\un v \un z_x$ in $\mathrm{HM}_i$ that we created in the proof of Lemma \ref{lemma_lower}, we define a witness in $\mathrm{HM}_n$:
Continue the code of a witness $\un v\un z_x$  in a way that the type is changed at every character (otherwise arbitrarily), obtaining the word $\un v\un z_x \un w_x $. One needs $n$ letters in total so that the concatenated word $\un w_x$ is of length $k-1$ with alternating types. Recall also that $\un z_x$ has length $n_1+1$. As a result, any witness vertex has $k-1+n_1+1$ many characters of alternating types at the end of its code.

It is left to show that the pairwise distance between any two vertices  is at least $2(k-1)+2n_1+1 \ge \ell_k$. 
There are two cases: namely, either the common prefix is of length $i$ or not. In the first case, the distance between $\un v\un z_x \un w_x$ and $\un v \un z_y \un w_y$ is at least $2 (k-1+n_1)+1\ge \ell_k$. This can be seen by the same argument as in the proof of Lemma \ref{lemma_lower}. Namely, any path that tries to connect two of these witnesses must change the type of the postfix at least $k-1+n_1$ times on both sides of the path and one needs an extra edge in the middle (since $x\neq y$), obtaining the required distance. 

When the common prefix is shorter, then $\un v\neq \un v'$, and the witnesses are of the form $\un v \un z_x \un w_x$ and $\un v' \un z_y\un w_y$ with possibly $x=y$. 
In this case point {\bf(a)} after Def.~\ref{edgeset_def}  applies and even if $x=y$, one needs to change the letters in the postfix one-by-one to obtain a postfix of length $n_1+1+k-1$ that has a type starting from both codes. This is at least $2(n_1+1+k-1)$ changes again, and there is at least $1$ extra edge necessary since the common prefix is shorter than $i$ characters. As a result the distance is again at least $\ell_k$.
\end{proof}

Now we can prove the existence of the transfinite fractal dimension of the hierarchical graph sequence model and determine the value of $\tau$.
\begin{proof}[Proof of Theorem \ref{hier_tau}]
The first statement follows since $N_B(\ell)/N$ in this case is not polynomial but exponential in $\ell$ by Lemma \ref{lemma_lower2}. For the  transfinite fractal dimension, we note that it is enough to determine a subsequential limit in $\ell$ along the sequence $\ell_k =\mathrm{diam}(\mathrm{HM}_k)+1$ since  $N_B^n(\ell)$ is monotone decreasing in $\ell$. Hence we have
$$\tau \left( \left\{ \mathrm{HM}_n\right\}_{n\in \mathbb N} \right) = \lim_{\ell \to \infty} \lim_{n \to \infty} \frac{\log \left(N_B^n(\ell) / |\mathrm{HM}_n| \right)}{-\ell}
= \lim_{k \to \infty} \lim_{n \to \infty} \frac{\log \left(N_B^n(\ell_k)/ |\mathrm{HM}_n| \right)}{-\ell_k}.$$
Recall that $|\mathrm{HM}_n|=N^n$ and from Lemma \ref{diam_lem} we have $\ell_k=\mathrm{diam}(G) -1+ 2k$ from \eqref{eq:ellk}. Using Lemma \ref{lemma_lower2} we give an upper bound on $\tau$:

\begin{align}\label{eq:upper}
\tau \left( \left\{ \mathrm{HM}_n\right\}_{n\in \mathbb N} \right) &\leq  \lim_{k \to \infty} \lim_{n \to \infty} \frac{\log\left( N^n\right) - \log \left( N^{n-k+1-n_1}\right)}{\mathrm{diam}(G) -1+ 2k} \nonumber \\
&= \lim_{k \to \infty} \frac{(k-1+n_1)\log N}{\mathrm{diam}(G) -1+ 2k} = \frac{\log N}{2}.
\end{align}
Similarly, Lemma  \ref{lemma_upper} yields a lower bound on the value of $\tau$:
\be \label{eq:lower}
 \tau \left( \left\{ \mathrm{HM}_n\right\}_{n\in \mathbb N} \right) \geq (\log N)/2.
\ee
Combining \eqref{eq:upper} and \eqref{eq:lower} we can conclude that $\tau \left( \left\{ \mathrm{HM}_n\right\}_{n\in \mathbb N} \right) = (\log N)/2$.
\end{proof}

\section{Song-Havlin-Makse model}
\label{sec:shm}

In this section, we analyze the model proposed by Song, Havlin and Makse in \cite{song2006origins} to generate graphs with and without fractal scaling of Eq. \eqref{powerbox}.
The motivation of the model is that the main feature that seems to distinguish the fractal networks is an effective ``repulsion'' (dissortativity) between nodes with high degree (hubs),
this idea was first suggested by Yook \emph{et al.} based on empirical evidence \cite{yook2005self} and developed by Song \emph{et al.} with analytical and modeling confirmations \cite{song2006origins}.
To put in other words, the most connected vertices tend to not be directly linked with each other but they prefer to link with less-connected nodes. In contrast, in case of non-fractal networks,
hubs are primarily connected to hubs. The model $\left\{\mathrm{SHM}^p_n \right\}_{n\in \mathbb N}$ defined below can capture the main features (e.g. scale-free \cite{molontay2015fractal}) of real-world networks and the presence of fractal scaling is governed by a
parameter of the model.

\begin{itemize}
	\item \textbf{Initial condition:} We start at $n=0$ with an arbitrary connected simple graph of a few vertices (e.g. a star shape of five nodes as in Figure \ref{fig:modes}).
	\item \textbf{Growth:} At each time step $n+1$ we link $m\cdot\mathrm{deg}_n(v)$ new vertices to every $v$ vertex that is already present in the network, where $m >1$ is an input parameter and $\mathrm{deg}_n(v)$ is the degree of vertex $v$ at time $n$.
	\item \textbf{Rewiring edges:} At each time step $n+1$ we rewire the already existing edges as a stochastic combination of Mode I (with probability $p$) and Mode II (with probability $1-p$)
	\begin{itemize}
		\item \textbf{Mode I:} we keep the old edge generated before time $n+1$.
	  \item \textbf{Mode II:} we substitute the edge $(u,v)$ generated in one of the previous time steps by a link between newly added nodes, i.e., by an edge $(u', v')$, where $u'$ and $v'$ are newly added neighbors of $u$ and $v$ respectively, as shown in Figure \ref{fig:modes}.
	\end{itemize}
\end{itemize}

\begin{figure}[h]
	\centering
		\includegraphics[width=90mm]{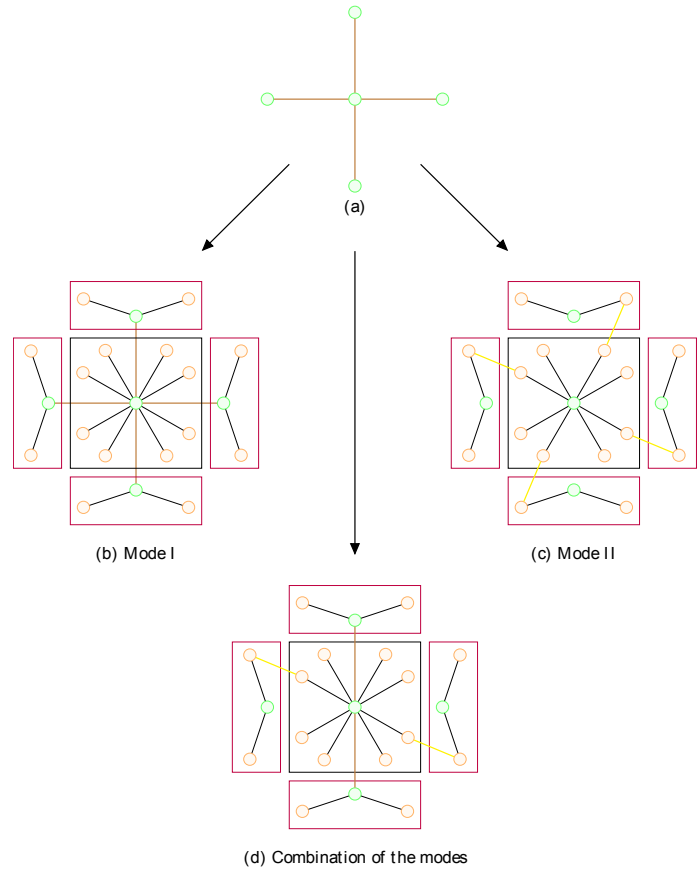}
	\caption{Different modes of growth with $m=2$. In (a) the initial stage is illustrated, (b), (c) and (d) demonstrates Mode I, Mode II and the combination of the two modes respectively. The figure was adapted by the authors from \cite{song2006origins}.}
	\label{fig:modes}
\end{figure}

The model evolves by linking new nodes to already existing ones as follows: those nodes that appeared in the earlier stages form the hubs in the network.
Consequently, Mode I leaves the direct edges between the hubs leading to hub-hub attraction, i.e. there are edges between vertices with high degrees. On the contrary Mode II leads to hub-hub repulsion or anticorrelation. It is interesting to investigate
how the connection mode affects the fractality of the model, what happens if only one of the modes is used ($p=0$ or $p=1$) or the combination of the two modes (i.e. $p \in (0,1)$).

Here we note that using only Mode I ($p=1$) results in a tree (assuming that the initial graph contained no cycles) but this tree is not a rooted locally finite tree in contrast with the trees considered in Section \ref{boxingtrees}. The model with parameter $p=0$ (i.e. using only Mode II) has a finite box dimension with $d_B= \log(2m+1)/\log 3$~\cite{song2006origins}. On the other hand, with parameter $p=1$ the box dimension is not finite but the  transfinite fractal dimension is still a valid parameter, that is the content of the next result:

\begin{theorem}\label{modeI}
The graph sequence $\left\{ \mathrm{SHM}^1_n\right\}_{n\in \mathbb N}$ is not fractal, but transfractal. That is, its fractal dimension (as in Def.~\ref{boxdim}) does not exists, while its transfinite fractal dimension (as in Def.~\ref{modboxdim}) exists and equals 
\be \label{tau_value-shm}
\tau \left( \left\{ \mathrm{SHM}^1_n\right\}_{n\in \mathbb N} \right)= \log (2m+1).
\ee
Further, the diameter of the network generation model with parameter $p=1$ is
\be \label{shm-diam}\mathrm{diam}(\mathrm{SHM}^1_n)=\Theta(n).
\ee
\end{theorem}

The first half of the assertion above was claimed in \cite{song2006origins} with heuristic explanation, here we give a more analytical argument and prove that the model is transfractal. 

\begin{proof}[Proof of Theorem \ref{modeI}]
Let $V(n)$ and $E(n)$ denote the number of nodes and edges in the network at time $n$, respectively.
Observe that $V(n)$ is deterministic, namely, it satisfies the recursion $V(n) = V(n-1) + 2m E(n-1)$. Assuming that $E(0)=V(0)=c$ for a fixed constant $c$ then
\be\label{shm-nov} V(n)=(2m+1)^n \cdot c. \ee

Combining \eqref{shm-nov} with \eqref{shm-diam} yields that  $\left\{ \mathrm{SHM}^p_n\right\}_{n\in \mathbb N}$ with parameter $p=1$ (i.e. using only Mode I) leads to a small-world network, i.e., the diameter of the graph grows proportionally to the logarithm of the number of  vertices.

Now, we show \eqref{shm-diam}.
If we use only Mode I ($p=1$), the diameter increases by 2 in every step, thus $\mathrm{diam}(\mathrm{SHM}^1_{n+1})=\mathrm{diam}(\mathrm{SHM}^1_n)+2$, i.e. $\mathrm{diam}(\mathrm{SHM}^1_k)=\mathrm{diam}(\mathrm{SHM}_0) +2k$.  

 In order to handle fractality we should examine the boxing of the network.
Note that Mode I yields to a tree-like structure: assuming that the initial graph was a tree, the network is a tree itself. Clearly only one $\ell_k$-box is enough to cover $\mathrm{SHM}^1_n$ if $n < k+1$.
Following from the hierarchical structure of the model, we cover $\mathrm{SHM}^1_n$  with $V(n-k-1)$ $\ell_k$-boxes, with $\ell_k:=\mathrm{diam}(\mathrm{SHM}^1_k)+1$ if $n-k-1 \geq 0$. Namely, an appropriate boxing of $\mathrm{SHM}^1_n$ with $\ell_k$-boxes if we choose the centers of the boxes as the vertices generated $(k+1)$ steps ago. Let $N^n_B(\ell_k)$ denote the minimum number of $\ell_k$-boxes needed to cover $\mathrm{SHM}^1_n$, so we have
\be \label{upper}
N^n_B(\ell_k) \leq \begin{cases} V(n-k-1) &\mbox{if } n \geq k+1\\
1 & \mbox{if } n < k+1. \end{cases}
\ee
Next, we turn to a lower bound.
In case of covering $\mathrm{SHM}^1_n$ with $\ell_k$-boxes, with $\ell_k = \mathrm{diam}(\mathrm{SHM}^1_k)+1$, we find $V\left(n-k-\left\lceil \mathrm{diam}(\mathrm{SHM}_0)/2\right\rceil\right)$ witness vertices such that
the pairwise distances between the vertices are greater than $\mathrm{diam}(\mathrm{SHM}^1_k)$. 
Namely, at time $n$, consider the vertices generated at time $n-k-\left\lceil \mathrm{diam}(\mathrm{SHM}_0)/2\right\rceil$ ago, and call these seeds. For each seed vertex, choose a \emph{descendant} of this vertex that was generated at step $n$ with distance of $k+\left\lceil \mathrm{diam}(\mathrm{SHM}_0)/2\right\rceil$ from the seed.  From the tree structure, any path that connects two of these witnesses must go to the seed of the witness vertex first. Hence, the path connecting two witnesses is at least $2(k+\left\lceil \mathrm{diam}(\mathrm{SHM}_0)/2\right\rceil) +1$ long, that is at least $\ell_k$, since  $2(k+\left\lceil \mathrm{diam}(\mathrm{SHM}_0)/2\right\rceil) +1 \geq 2k+\mathrm{diam}(\mathrm{SHM}_0) +1 = \mathrm{diam}(\mathrm{SHM}^1_k) +1 = \ell_k$. Hence, we have

\be \label{lower}
N^n_B(\ell_k) \geq \begin{cases} V\left(n-k-\left\lceil \mathrm{diam}(\mathrm{SHM}_0)/2\right\rceil\right) &\mbox{if } n \geq k+\left\lceil \mathrm{diam}(\mathrm{SHM}_0)/2\right\rceil\\
1 & \mbox{if } n < k+\left\lceil \mathrm{diam}(\mathrm{SHM}_0)/2\right\rceil. \end{cases}
\ee
 Therefore, combining \eqref{upper} and \eqref{lower} we can conclude that $N^t_B(\ell_k) = \Theta\left(V(n-k)\right)$,
this together with the exponential growth of $V(n)$ and the linear grow of $\mathrm{diam}(\mathrm{SHM}^1_n)= \mathrm{diam}(\mathrm{SHM}_0)+2n = \Theta(n) $ yields that no finite
$d_B$ exists in the sense of Def.~\ref{boxdim}, thus Mode I leads to a small-world non-fractal topology.

On the other hand, we can consider the transfinite fractal dimension $\tau$ along the subsequence of box sizes $\ell_k = \mathrm{diam}(\mathrm{SHM}^1_k)+1=\mathrm{diam}(\mathrm{SHM}_0)+2k+1$ due to the monotonicity of $N_B^n(\ell)$ in $\ell$. It is clear from \eqref{upper} and from \eqref{lower} combined with the fact that $V(n)=\Theta\left((2m+1)^n\right)$ that

$$\tau\left(\left\{ \mathrm{SHM}^1_n\right\}_{n\in \mathbb N}\right) =  \lim_{k \to \infty} \lim_{n \to \infty} \frac{\log\left( (2m+1)^{n-k-1}/ (2m+1)^n \right)}{-\ell_k} = \log(2m+1). $$
We can conclude that the transfinite fractal dimension of the graph sequence $\{\mathrm{SHM}^1_n\}$ is $\log (2m+1)$.
\end{proof}

\begin{remark}
Unfortunately, our current techniques are not able to handle the case when Mode II is also present. Heuristically, Mode II destroys this exponential growth so much that the growth rate becomes polynomial. As further research, it would be interesting to study rigorously the interpolation between these two very different growth rates by studying where the phase transition takes place between the two regimes. We pose the following open question: Is there such a $p_c \in [0,1]$ such that $\tau \left( \left\{ \mathrm{SHM}^p_n\right\}_{n\in \mathbb N} \right)$ exists and nonzero for all $p \geq p_c$ while the usual box dimension exists for all $p < p_c$? It was claimed in \cite{song2006origins} with heuristic explanation that $p_c=1$.
\end{remark}


\section{Boxing of trees and connection to the growth rate}
\label{boxingtrees}

In this section, we calculate the transfinite fractal dimension $\tau$ for some rooted evolving trees and compare it to the value of growth rate defined by Lyons and Peres \cite{lyons2016probability}. We find that the transfinite fractal dimension of infinite trees is strongly related to the growth rate.  While growth rate is defined only on trees, our concept of transfinite fractal dimension is defined on any graph sequences. 
Next, we define the growth rate, and to be able to do so, we need some notation.

Let us denote an infinite tree by $T_{\infty}$, and its root by $\varrho$. We assume that  the degree of each vertex is finite. 
Let $\CL_n$ denote the set of vertices at distance $n$ from the root, and its size by $L_n$. 
\begin{definition}[Growth rate of trees, \cite{lyons2016probability}]
\label{def::growthrate_def} The growth rate of an infinite tree  $T_\infty$ is defined as
\begin{equation}\label{def::gr}
\mathrm{gr} \left(T_\infty\right) =\lim_{n \to \infty} L_n^{1/n},
\end{equation}
whenever this limit exists.
\end{definition}

Note that an infinite tree needs infinitely many boxes of any size. To be able to define a proper transfinite dimension, we need to `chop off' the tree to make it finite. We also provide some definition of basic expression related to trees.
\begin{definition}[Basic definitions]Consider a rooted and infinite tree $T_{\infty}$.
We obtain a sequence $\left\{ T_n\right\}_{n\in \mathbb N}$ from $T_{\infty}$ by truncating at height $n$:
\be T_n:=\cup_{i\le n}\CL_i\ee
We define the (transfinite) fractal dimension of an infinite tree as the (transfinite) fractal dimension of the graph sequence $(T_n)_{n\ge 1}$.

The \emph{generation} of a vertex is its graph distance from the root. The \emph{subtree} of a vertex $v$, denoted by $T^{(v)}$, is defined as the vertices $w$ that have the property that the shortest path to the root passes through $v$. The descendants of $v$ are the vertices in $T^{(v)}$. We write $T_k^{(v)}=\{w\in T^{(v)}: d(v,w)\le k\}.$ The vertices in $T_1^{(v)}\setminus\{v\}$ are called the children of $v$.

\end{definition}
Observe that the diameter of $T_k^{(v)}$ is at most $2k$, hence, $T_k^{(v)}$ is a $2k+1$-box.

We shall use the notation introduced in the definition above throughout the rest of paper. In the rest of this section we investigate the optimal boxing of various trees for certain box sizes -- along the subsequence $\ell_k:=2k+1$ and we write $N_B^n(\ell_k)$ for the minimal number of $\ell_k$-boxes that we need to cover $T_n$.
We compare the transfinite fractal dimension and the growth rate of some trees. Let us start with some examples; later we will generalise them to spherically symmetric trees.

\begin{example}[Complete $d$-ary tree]\label{d-ary}
 A complete $d$-ary tree $T_{\infty}^d$ is a rooted tree where each vertex has exactly $d$ ($d \geq 2$) children.
\end{example}

\begin{example}[``2-3''-tree]\label{2-3}
 A ``2-3''-tree $T_{\infty}^{2,3}$ is a rooted tree such that vertices at even distances from the root have 2 children while all other vertices have 3 children \cite{lyons2016probability}.
\end{example}

An important tool for  the boxing of trees is the greedy boxing method.

\begin{definition}[Greedy boxing starting from the leaves]\label{greedy} Let  $n+1:= a (k+1) +b$ for some $a,b \in \mathbb{N}, b < k+1.$ We define the greedy boxing of a rooted tree $T_n$ with $2k+1$-boxes as follows: 
\be \mathrm{Greedy}(n,k):=\left(\bigcup_{i=1}^{a} \bigcup_{v\in \CL_{n+1-i(k+1)}} T_k^{(v)} \right)\cup T_{b-1}^{(\varrho)}.\ee
That is, every vertex $v$ at generation $n+1-i(k+1)$, for $i= 1, \dots, a$ and its subtree  $T_{k}^{(v)}$ forms one box, and the box of the root might be somewhat smaller. When $b=0$ the box of $\varrho$ is included in the first union hence the last box ($T_{b-1}^{(\varrho)}$) is not there in the expression.
\end{definition}

\begin{figure}[h]
	\centering
		\includegraphics[width=120mm]{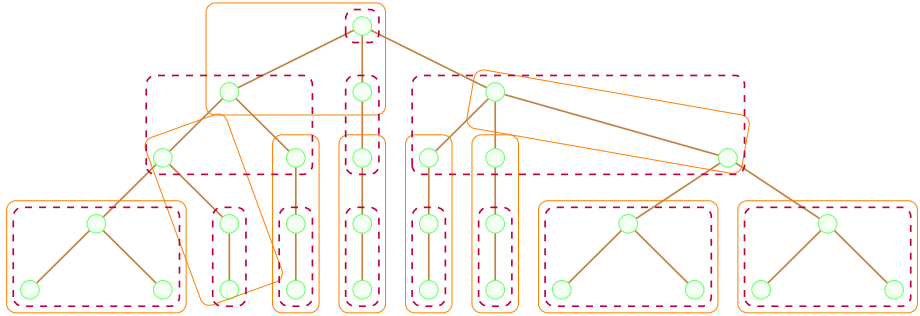}
	\caption{The boxing of rooted trees. The dashed boxes illustrates the greedy boxing and the orange boxes show an optimal boxing with box size $\ell=3$. Using the notation of Def.~\ref{greedy}, here $n=4$, $k=1$, $a=2$, $b=1$.}
	\label{fig:greedy}
\end{figure}

 \begin{lemma}\label{claim:greedy} Consider a rooted tree $T_n$, where each vertex not in generation $n$ has at least one child.  $N_B^n(\ell_k)$, the minimal number of $(2k+1)$-boxes, satisfies
\be L_{n-k}\le N_B^n(\ell_k)\le \mathrm{G}(n,k)\le \sum_{i=0}^{n-k}L_{i},\ee
where $\mathrm{G}(n,k)$ is the number of boxes in the greedy boxing, and $L_{i}$ is the size of generation $i$.
\end{lemma}
\begin{proof}
The fact that $N_B^n(\ell_k)\le \mathrm{G}(n,k)$ follows from the optimality of $N_B^n(\ell_k)$. The last inequality is established by observing that $\mathrm{G}(n,k)$ uses vertices in $\CL_{n-k}$ to cover generations $n-k, \dots, n$. The sum of earlier generations is a somewhat crude upper bound, since not all vertices in every generation before $n-k$ are put in a separate box. 
  
 For the lower bound, we find $L_{n-k}$ witness vertices, that is vertices with pairwise distance at least $2k+1$ away. We choose a witness vertex from each box $T_k^{(v)}$, $v\in \CL_{n-k}$: let $w(v)$ be any of the vertices in $T_k^{(v)}$ that is in generation $n$. 
We show that when $u\neq v$ are two distinct witnesses, then $d(w(u), w(v))\ge 2k+1$. Indeed, since $w(u), w(v)$ are in different subtrees in generation $n-k$, the shortest path between $w(u), w(v)$ travels through both $u$ and $v$ and so it contains at least $d(w(u),u)+d(w(v),v)+2=2k+2$ edges.
\end{proof}

\begin{theorem}

The complete $d$-ary tree and ``2-3''-tree have finite  transfinite fractal dimension, namely $\tau\left(T_{\infty}^{d}\right)=(\log d) / 2$ and $\tau\left(T_{\infty}^{2,3}\right)=(\log\sqrt6) / 2$.
The growth rate of these trees are $d$ and $\sqrt6$ respectively.
\end{theorem}
\begin{proof}
A direct application of Def.~\ref{def::growthrate_def} yields that $ \mathrm{gr} \left( T_{\infty}^d \right) = d$. For the  transfinite fractal dimension, by monotonicity it is enough to consider the subsequence $2k+1$ for $\ell$. Applying Lemma \ref{claim:greedy} for this case yields that 
\be d^{n-k} \le N_B^n(\ell_k) \le \sum_{i=0}^{n-k} d^{i} = \frac{d^{n-k+1}-1}{d-1}.\ee
Because $\left|V(T_n^d)\right| = (d^{n+1}-1)/(d-1)$, using the bounds from Lemma  \ref{claim:greedy} we can calculate the limit
$$\tau\left(T_d^\infty\right) = \lim_{k \to \infty} \lim_{n \to \infty} \frac{\log( N_B^n(\ell_k)/\left|V(T_n)\right|)}{-(2k+1)} = \frac{\log d}{2}.$$
The proof works similarly for the ``2-3''-tree. It is shown in \cite{lyons2016probability} that $ \mathrm{gr} \left( T_{\infty}^{2,3} \right) = \sqrt6$. Elementary calculation shows that
\be \left|V(T_n^{2,3})\right| = \begin{cases}
3 \cdot (6^{(n+1)/2}-1)/5 \quad &\text{ if $n$ is odd}\\
3 \cdot (6^{n/2}-1)/5+6^{n/2} \quad &\text{ if $n$ is even.}
\end{cases} \ee
A simple calculation yields that 
\be L_n= \begin{cases}
2 \cdot 6^{(n-1)/2}\quad  &\text{if $n$ is odd}\\
6^{n/2} \quad &\text{if $n$ is even.}
\end{cases} \ee
Combining the results above with Lemma  \ref{claim:greedy} yields that
$$ \tau\left( T_\infty^{2,3}\right) = (\log\sqrt6) / 2. $$
\end{proof}

\noindent Observe that for these examples, the relation  $\tau\left( T_\infty \right)  = \ln{ \mathrm{gr} \left( T_\infty \right)  } / 2$ holds. The question naturally arises that under which conditions it is true that this relation is valid? We will answer this question in the following sections.

\subsection{Spherically symmetric trees}
\label{spherical}

Let  $T_{\infty}^{\mathbf f}$ be a spherically symmetric tree, that is an infinite rooted tree such that for each $h$, every vertex at distance $h$ from the root has the same number of children, namely,  $f(h)$ many, $f(h) \in \mathbb{N}^+=\{1,2,\dots\}$. Examples \ref{d-ary} and \ref{2-3} are also spherically symmetric trees.
For spherically symmetric trees, $L_n= \prod_{h=0}^{n-1}f(h)$ and $\left|T_n^f\right|=\sum_{i=0}^n L_i$. 

In what follows, we investigate the assumptions needed on the sequence  $\mathbf f=(f(h)_{h\in \mathbb N})$ so that the  transfinite fractal dimension of a spherically symmetric tree exists and equals the half of the logarithm of the growth rate. This question is nontrivial, as demonstrated by the following examples. 

\begin{example} \label{ex1}
Let $$f(i) = \begin{cases} \left\lfloor e^j \right\rfloor  & \text{ if $ i=j!$ for some j} \\ \text{ fixed } c \in \mathbb{N}  & \text{ otherwise.}\end{cases}$$
\end{example}
\begin{example} \label{ex2}
Let
 $$\log f(h) := \begin{cases} a \quad  & \text{ if $ m(m+1)<h \leq (m+1)^2$ for some $m \in \mathbb{N}$} \\ b \neq a  \quad  & \text{ otherwise.}\end{cases}$$
 \end{example}
In Example \ref{ex2} we have blocks of $'a'$s and blocks of $'b'$s with linearly increasing lengths, while in Example \ref{ex1} $f$ takes on a very large value `occasionally'. We show below that while the growth rate of these trees exists, the transfractal dimension is not even defined, i.e., the limit in Def.~\ref{modboxdim} does not exists. First, we need the following definition.

\begin{definition}[Length of the maximum contiguous subsequence of the same element]
Let $\mathbf{s}=\left(s_i \right)_{i = 0}^{\infty}$ be an arbitrary sequence with codomain\footnote{The codomain of a sequence is the set into which all of the elements of the sequence is constrained to fall.} $S$ and $a\in S$. The length of the maximum contiguous subsequence of sequence $\mathbf s$ with respect to element $a$ (i.e., the length of the longest block of consecutive $'a'$s:
$$ \mathrm{lmcs}(\mathbf s,a):= \max\{n \in \mathbb N \mid \exists i\in \mathbb N: a=s_i=s_{i+1}=s_{i+2}=\dots =s_{i+n}\}+1.  $$
\end{definition}

\begin{condition}[Regularity assumption]\label{cond::reg}
Let us assume that $\mathbf f=(f(h))_{h\ge 0}$ in a spherically symmetric tree $T_\infty^{\mathbf f}$ satisfies for some $K\in \mathbb{N}$ that
\be  \mathrm{lmcs}(\mathbf f,1)=K  <\infty .\ee
\end{condition}
\begin{lemma} \label{lemma_c1}
Condition \ref{cond::reg} implies that for all $n\ge 1$ it holds that
\be\label{eq:generation-bound} L_n\le |T_n^{\mathbf f}|=\sum_{i=0}^n L_i \le  2(\mathrm{lmcs}(\mathbf f,1) +1)\cdot L_n, \ee
that is, the total size of the tree is the same order as the size of the last generation.
\end{lemma}
\begin{proof}[Proof of Lemma \ref{lemma_c1}]
The first inequality of the lemma is trivial.
Now we express $L_i$  in terms of $L_n$:

\be\label{eq:gen_ub} L_i =  L_n \prod_{h=i}^{n-1} f(h)^{-1}  \le L_n \prod_{h=i}^{n-1} f^*(h)^{-1}, \ee
where we define $\mathbf f^\star:=\mathbf f^\star(\mathbf f)$ as 
$$f^*(h) := \begin{cases}
1,  & \text{if } f(h)=1\\
2,  & \text{if } f(h) \ge 2.
\end{cases}$$
By denoting $\lfloor x \rfloor:=\sup\{y\in \mathbb Z: y\le x\}$ the lower integer part of a real number $x$, and $K:=\mathrm{lmcs}(\mathbf f,1)$, 
\be\label{eq:f_ub}
\prod_{h=i}^{n-1} f^*(h)^{-1} \le \left( \frac{1}{2}\right)^{\left\lfloor (n-i)/(K+1) \right\rfloor}
\ee
since the product is maximized when $\mathbf f^\star$ is such that $K$ 1's are followed by a single $2$ periodically.
Combining \eqref{eq:gen_ub} and  \eqref{eq:f_ub} yields that
\be
\sum_{i=0}^n L_i \le L_n \left( 1 + \sum_{i=0}^{n-1}\left( \frac{1}{2}\right)^{\left\lfloor (n-i)/(K+1) \right\rfloor} \right) 
  \le L_n (K+1) \sum_{i=0}^\infty 2^{-i},\ee

\noindent since $\left\lfloor(n-i)/(K+1)\right\rfloor$ takes on each integer $ \le \left\lfloor n/(K+1) \right\rfloor$ at most $K+1$ times. The sum of the geometric series is at most $2$ and thus \eqref{eq:generation-bound} is established. 

\end{proof}

\begin{remark}
Condition \ref{cond::reg} is not only sufficient but necessary for Lemma \ref{lemma_c1}. If $\mathrm{lmcs}(\mathbf f,1) > M,$ $\forall M \in \mathbb{N}$, then for some $n$, $f(n-M) = \dots = f(n) = 1$; and then $\sum_{i=n-M}^n L_i =  (M+1) L_n.$ 

\end{remark}

\begin{corollary} \label{cor:tau}
Under Condition \ref{cond::reg} the transfinite fractal dimension and growth rate of a spherically symmetric tree $T_{\infty}^f$ can be expressed as follows, if the limits exist:
\begin{align}\label{eq::tau-1}
\tau\left( T_{\infty}^f\right)&=
\lim_{k \to \infty} \lim_{n \to \infty} \frac{\sum_{h=n-k}^{n-1}\log f(h)}{(2k+1)},\\
\log \left(\mathrm{gr} \left( T_{\infty}^f\right)\right) &= \lim_{n \to \infty} \frac1n  \sum_{h=0} ^{n-1} \log f(h).\label{eq::gr-1}
\end{align}
\end{corollary}
\begin{proof} [Proof of Corollary \ref{cor:tau}]
Let us write $c_+:= 2(\mathrm{lmcs}(\mathbf f,1)+1)$. Combining Lemma \ref{claim:greedy} with Lemma \ref{lemma_c1} yields
 \be\label{eq::boxing-bound1} L_{n-k}\le N_B^n(\ell_k)\le \sum_{i=0}^{n-k}L_{i} \le c_+ L_{n-k}.\ee
Using Lemma \ref{lemma_c1} once more yields that $L_n\le |T_n^{\mathbf f}|\le c_+ L_n$ and combining this with \eqref{eq::boxing-bound1} we can observe
\be\label{eq:ratio}  \frac{N_B^n(\ell_k)}{|T_n^{\mathbf f}|}\in [1/c_+, c_+] L_{n-k}/ L_n = [1/c_+, c_+] \prod_{h=n-k}^{n-1} f(h)^{-1}.\ee
Using these bounds in Def.~\ref{modboxdim} we see that the constant prefactor will vanish when taking logarithm, yielding \eqref{eq::tau-1}. Eq. \eqref{eq::gr-1} is a direct consequence of Def.~\ref{def::growthrate_def}. 

\end{proof}

With Corollary \ref{cor:tau} at hand, we arrive to:
\begin{claim}The growth rate of Examples \ref{ex1} and \ref{ex2} exists, while they are not transfractal.
\end{claim}
\begin{proof}
We start with Example \ref{ex2}. The limit in \eqref{eq::gr-1} exists, since it equals
\[\log \mathrm{gr}(T_{\infty}^{\mathbf f})= \lim_{m\to \infty}\frac1{m^2} \sum_{h=1}^m ((h+1)a + h b)=(a+b)/2. \]
On the other hand, the inner limit (as $n\to \infty$) on the RHS of \eqref{eq::tau-1} does not exist for fixed $k$, since we see oscillations of $a$'s and $b$'s when $k\le m$, thus $\tau$ is undefined.
For Example \ref{ex1},
\[\lim_{n\to \infty} \frac1n\left( \sum_{i\le n} \log(c)- \sum_{j: j!\le n} j\right)\le \log \mathrm{gr}(T_{\infty}^{\mathbf f})\le  \lim_{n\to \infty} \frac1n\left(\sum_{j: j!\le n} j + \sum_{i\le n} \log(c)\right), \]
and hence the growth rate exists and equals $c$ since both the lower bound as well as the upper bound tend to $\log c$. On the other hand, the transfinite dimension does not exist, since, for each fixed $k$, as $n$ grows the inner sum $\sum_{h=n-k}^{n-1}\log f(h) $ occasionally encounters a factorial and hence a value other then $\log c$ as one of its terms, hence, the inner limit as $n\to \infty$ does not exist.
\end{proof}

Example \ref{ex2} led us to a natural generalization of the transfinite fractal dimension in such a way that it agrees the half of the logarithm of the growth rate for spherically symmetric trees under a mild condition on the growth of the degree sequence. This is the transfinite Cesaro fractal dimension in Def.~\ref{def::ourdim_def2}.

\begin{theorem}
Let $T_{\infty}^{\mathbf f}$ be a spherically symmetric tree with degrees $\mathbf f=(f(h))_{h\ge 1}$ that satisfies Condition \ref{cond::reg} and that
 \be \label{thm:ass}  \lim_{h\to \infty} \frac{\log f(h)}{h}=0.\ee Then, $\tau^\star(T_{\infty}^{\mathbf f})=\log (\mathrm{gr}(T_{\infty}^{\mathbf f}))/2$.
\end{theorem}
\begin{remark}The growth condition \eqref{thm:ass}  means that the degrees grow sub-exponentially. For instance, $f(h)=\lfloor\exp(h^\gamma)\rfloor$ for any $\gamma<1$ satisfies this criterion.
\end{remark}
\begin{proof}
By Def.~\ref{def::ourdim_def2} we obtain for any spherically symmetric tree satisfying Condition \ref{cond::reg}, using \eqref{eq:ratio}
\be\label{eq:cesaro-1}
	\tau^*\left(T_{\infty}^f\right)=
\lim_{k \to \infty} \lim_{m \to \infty}  \frac1m \sum_{i=1}^m \frac{\log(B_k^{i+k}/\left|T_{i+k}^{\mathbf f}\right|)}{-(2k+1)}
=
\lim_{k \to \infty} \lim_{m \to \infty}  \frac1m \sum_{i=1}^m \frac{\log(L_{i}/L_{i+k})}{-(2k+1)}.
\ee
Observing that $L_n=\prod_{h=1}^{n-1}f(h)$  and $\lim_{k \to \infty} (2k)/(2k+1)=1$ we obtain that
\be\label{eq:3case}
\tau^*\left(T_{\infty}^f\right)=
\lim_{k\to \infty}\lim_{m\to \infty} \frac1{2mk}\sum_{i=1}^m \sum_{j=i}^{i+k-1} \log f(j).
\ee
Exchanging sums yields that the RHS inside the limits equals
\be\label{eq::3terms}
  \frac{1}{2m} \sum_{j=1}^{k-1} \frac{j}{k} \log f(j) +  \frac{1}{2m}\sum_{j=k}^{m-1}  \log f(j)
+  \frac{1}{2m} \sum_{j=m}^{m+k-1}\left( 1- \frac{j-m}{k}  \right) \log f(j).
\ee
%
The first term in \eqref{eq::3terms} tends to zero as evaluating the inner limit ($m\to \infty$) in \eqref{eq:cesaro-1} for any fixed $k$.
For the second term in \eqref{eq:3case} we have:
\be
\begin{aligned}
\label{eq::2ndterm}
	\lim_{k \to \infty} \lim_{m \to \infty} \frac{1}{2m} \sum_{j=k}^{m-1}  \log f(j) &= \lim_{k \to \infty} \lim_{m \to \infty} \frac{1}{2m} \sum_{j=1}^{m-1} \log f(j) - \lim_{k \to \infty} \lim_{m \to \infty} \frac{1}{2m} \sum_{j=1}^{k-1} \log f(j) \\
&=    \lim_{m \to \infty} \frac{1}{2m} \sum_{j=1}^{m-1} \log f(j),
\end{aligned}
\ee
since the second term tends on the RHS to zero for each fixed $k$, and the first term does not depend on $k$.
Regarding the third term in \eqref{eq:3case}, assumption \eqref{thm:ass} implies that for any fixed $k$,
\be\label{eq::3trdterm} \lim_{m \to \infty} \frac{1}{2m} \sum_{j=m}^{m+k-1}\left( 1- \frac{j-m}{k}  \right) \log f(j)= 0, 
\ee
 since the bounds $0 \leq \left( 1- \frac{j-m}{k}  \right) \log f(j) \leq  \log f(j)$ gives the desired result using the squeeze theorem. The question arises whether condition \eqref{thm:ass} could be weakened to the condition in \eqref{eq::3trdterm}. For this, note that each term in \eqref{eq::3trdterm} must tend to zero. The term with  $j=m$ equals  $(\log f(m))/2m$, that is $1/2$ of the term in \eqref{thm:ass}. In other words, the limit \eqref{eq::3trdterm} being equal to zero is equivalent to \eqref{thm:ass}.

Combining \eqref{eq::3terms}, \eqref{eq::2ndterm}, \eqref{eq::3trdterm} with Corollary \ref{cor:tau} we obtain that 
$$ \tau^*\left(T_{\infty}^f\right) = \lim_{m \to \infty} \frac{1}{2m} \sum_{j=1}^{m-1} \log f(j)=\frac{\log \mathrm {gr}(T_{\infty})}{2}.$$

\end{proof}
\subsection{Supercritical Galton-Watson trees}

In this section, we will consider the supercritical Galton-Watson trees.
\begin{definition}[Galton-Watson branching process]
Let $\mathbf q = (q_0, q_1, q_2, \dots)$ be an infinite vector of nonnegative real numbers with $\sum_{i=0}^{\infty}q_i=1$ and let $\mathbb{P}_{\mathbf q}$ be the probability measure
on rooted trees such that the number of offspring (children) of each vertex is i.i.d. and the probability that a given
vertex has $i$ children is $q_i$.

Let $Z_n$ denote the number of vertices at distance $n$ from the root. We write $\mu$ for the mean of the offspring distribution $\mu = \mathbb{E}(Z_1) = \sum_{i=0}^{\infty}iq_i$.

A Galton-Watson branching process (BP) with mean offspring $\mu$ is said to be supercritical
if $\mu > 1$, critical if $\mu = 1$ and subcritical if $\mu < 1$.
\end{definition}
It is well-known that in the subcritical case ($\mu < 1$) and in the critical case ($\mu = 1$), the BP dies out eventually (i.e. $\exists \; n: \; Z_n=0$) with probability $1$, while in the supercritical
case ($\mu > 1$), the BP survives (i.e., $\forall \; n: \; Z_n\ge1$) with positive probability \cite{athreya2012branching}. When $q_0=0$, the BP survives with probability $1$, since each vertex has at least one child. For further discussion and characterization of Galton-Watson BPs we refer the reader to \cite{athreya2012branching} and \cite{duquesne2002random}.
A supercritical  Galton-Watson BP behaves similarly to a  deterministic regular tree of the ``same growth'' that suggests that the  transfinite fractal dimension should be $\log \mu$
where $\mu$ is the mean of the offspring distribution.
When $\mu$ is an integer, this deterministic tree is just the $\mu$-ary tree, but when $\mu$ is nonintegral, it is a virtual tree, in the sense of Pemantle and Peres \cite{pemantle1995critical,pemantle1995galton}.

\begin{theorem}
\label{G_W_tau}
Let $\mathrm{GW}_{\infty}$ be a supercritical Galton-Watson tree with the following assumptions
$ q_0 = 0, \; q_1  \neq 1 $
and
\be\label{xlogx} 
\mathbb{E}(Z_1 |\log(Z_1)|^+) < \infty,
\ee
where $|x|^+:=\max\{x, 0\}$ denotes the positive part of the expression.
Then, $ \tau \left( \mathrm{GW}_{\infty} \right)  \rightarrow (\log (\mu))/2$ almost surely.
\end{theorem}

The following elementary fact will be needed to prove Theorem \ref{G_W_tau}.

\begin{lemma} \label{GW_lemma} Let $\xi_k$ be an arbitrary convergent sequence of real numbers,  that is, $\lim_{k\to \infty}\xi_k =A$ for some $ A \in \mathbb{R}$. For $\beta\in(0,1)$ let $S_n:=\sum_{k=0}^n \beta^{n-k} \xi_k$. Then $$\lim_{n \to \infty} S_n = \sum_{k=0}^{\infty}\beta^k \cdot A =\frac{A}{1-\beta}\in \mathbb{R}.$$
\end{lemma}
\begin{proof}
Using the convergence of $\xi_k$, by definition $\forall \epsilon >0: \; \exists \; N^*$ s.t. if $i>N^*$ then $|\xi_i-A|<\epsilon$. Now we consider the following sum
$$\sum_{i=N^*}^n \xi_i \beta^{n-i}  \in (A-\epsilon, A+\epsilon) \cdot \sum_{k=0}^{n-N^*} \beta^k.$$
Hence,
\be \label{sum1}
\lim_{n \to \infty} \sum_{i=N^*}^{n} \xi_i \beta^{n-i} \in  (A-\epsilon, A+\epsilon) \cdot \sum_{k=0}^{\infty}\beta^k  =  \left(\frac{A-\epsilon}{1-\beta},  \frac{A+\epsilon}{1-\beta}\right).
\ee
Now let us consider the remaining part of the sum, namely $\sum_{i=0}^{N^\star-1} \xi_i \beta^{n-i}$. Due to the fact that $\xi_i$ is convergent thus for some $K: \xi_i <K, \; \forall i$ and $\lim_{n\to \infty}\sum_{k=n-N^*}^n\beta^k = 0$ we have
\be \label{sum2}
\sum_{i=0}^{N^*} \xi_i \beta^{n-i} < K \cdot \sum_{k=n-N^*}^n\beta^k \xrightarrow{n\to \infty} 0.
\ee
Using \eqref{sum1} and \eqref{sum2} we can conclude that $\lim_{n \to \infty} S_n = A/(1-\beta)$.
\end{proof}

\begin{proof}[Proof of Theorem \ref{G_W_tau}]
By writing $W_i:=Z_i \mu^{-i}$, \cite[Chapter I, Part C, Theorem 1]{athreya2012branching}, the limit $W=\lim_{i\to \infty}W_i$ exists and is in $(0, \infty)$ almost surely under the $\mathbf{x\log x}$ assumption in \eqref{xlogx} and that $q_0=0$. With this notation, we bound the ratio $B^n_k/\left|V(\mathrm{GW}_n)\right|$ using Lemma  \ref{claim:greedy} (note that $Z_i$ has the notation $L_i$ there):
\be  \label{eq:gw:1}
\frac{W_{n-k}\mu^{n-k}}{\sum_{i=0}^n W_i \mu^i} =\frac{Z_{n-k}
}{\sum_{j=0}^n Z_j}
 \le \frac{B^n_k}{\left|V(\mathrm{GW}_n)\right|} \leq  \frac{\sum_{i=0}^{n-k} Z_{i}
}{\sum_{j=0}^n Z_j}=\frac{\sum_{i=0}^{n-k} W_{i} \mu^{i}}{\sum_{i=0}^n W_i \mu^i}.
\ee
Dividing both the numerators and denominators with $\mu^{-n}$, the bound turns into
\be (1/\mu)^{k}\frac{W_{n-k} }{\sum_{i=0}^n W_i (1/\mu)^{n-i}} \le \frac{B^n_k}{\left|V(\mathrm{GW}_n)\right|} \leq (1/\mu)^{k} \frac{\sum_{i=0}^{n-k} W_i (1/\mu)^{n-k-i}}{\sum_{i=0}^n W_i (1/\mu)^{n-i}} \ee
Applying now Lemma \ref{GW_lemma} with $\xi_k:=W_k, \beta:=1/\mu$, we see that the almost sure limits exists for each fixed $k$
\be\ba \lim_{n\to \infty}& \frac{\sum_{i=0}^{n-k} W_i (1/\mu)^{n-k-i}}{\sum_{i=0}^n W_i (1/\mu)^{n-i}}=\frac{W/(1-1/\mu)}{W/(1-1/\mu)}=1,\\
\lim_{n\to \infty}& \frac{W_{n-k} }{\sum_{i=0}^n W_i (1/\mu)^{n-i}} =\frac{W}{W/(1-(1/\mu))}=1-(1/\mu).\ea \ee
Therefore, for almost surely
\be (1/\mu)^{k}(1-(1/\mu))/2 \le \frac{B^n_k}{\left|V(\mathrm{GW}_n)\right|} \leq 2 (1/\mu)^{k}.  \ee

Substituting these bounds into Def.~\ref{modboxdim} of transfinite fractal dimension, with $\ell=2k+1$, yields that
\be
\tau\left(\left\{ \mathrm{GW}_n\right\}_{n\in \mathbb N}\right) = \lim_{k \to \infty} \lim_{n \to \infty} \frac{\log(B^n_k/\left|V(\mathrm{GW}_n)\right|)}{-(2k+1)} = \frac{\log (\mu)}{2}.
\ee

\end{proof}

\begin{remark}
It is clear that $ \mathrm{gr} \left( \mathrm{GW}_{\infty} \right)  \xrightarrow{a.s.} \mu$ which implies that the almost sure limit of the transfinite fractal dimension and the logarithm of the growth rate of Galton-Watson trees differ only in a factor of 2 under the assumptions of Theorem \ref{G_W_tau}.
\end{remark}

\begin{remark}
Under some regularity assumptions the transfinite fractal dimension of Galton-Watson trees is well-defined in contrast with spherically symmetric trees where we had to introduce the concept of  transfinite Cesaro transfinite fractal dimension. This phenomena arises from the fact that the random growth of the Galton-Watson branching process is smoother than the deterministic growth of spherically symmetric trees.
\end{remark}

\section{Conclusion and discussion}
\label{conc}

In this paper, we have investigated the heuristic statement that networks with hierarchical
structure are fractal (i.e., self-similar) objects. In particular, we considered
a graph sequence with strict hierarchical structure, and investigated its fractal
properties. Doing so we showed that the definition of fractality cannot be applied
to networks with locally 'tree-like' structure and exponential growth rate of neighborhoods.
However, the box-covering method gives a parameter that is related to the growth rate of trees. We also introduced a more general concept, the transfinite Cesaro fractal dimension. We investigated various models: the hierarchical graph sequence
model introduced by Komj\'athy and Simon, Song-Havlin-Makse model,
spherically symmetric trees, and supercritical Galton-Watson trees. We determined bounds on the optimal box-covering and calculated the transfinite fractal dimension of the aforementioned models using rigorous techniques. It would be also interesting to apply our method to other locally tree like graphs such as Erd\H{o}s-R\'enyi graph, preferential attachment graph or configuration model.

\section*{Funding}

	The research reported in this paper was supported by the Higher Education
Excellence Program of the Ministry of Human Capacities in the frame of Artificial Intelligence research area of Budapest University of Technology
and Economics (BME FIKP-MI/SC). The publication is also supported by the EFOP-3.6.2-16-2017-00015 project entitled ”Deepening the activities of HU-MATHS-IN, the Hungarian Service Network for Mathematics in Industry and Innovations” through University of Debrecen. The project has been supported by the European Union, co-financed by the European Social Fund. The work of K. Simon and R. Molontay is supported by NKFIH K123782 research grant and by MTA-BME Stochastics Research Group. The work of J. Komj\'athy is partially financed by the
programme Veni \#639.031.447, financed by the Netherlands Organisation for Scientific Research (NWO).

	\section*{Acknowledgement} We would like to thank J\'anos Kert\'esz for useful conversations. We also thank Marcell Nagy for reading through the manuscript. We are grateful for the anonymous reviewers for their careful reading of our manuscript and their many insightful comments and suggestions.

\section*{Appendix}
\begin{proof}[Proof of Lemma \ref{diam_lem}]The proof is a rewrite of \cite{komjathy2011generating} that we include for completeness.
For two arbitrary vertices $\un x,\un y \in \Sigma_n$ we denote the length of their common prefix by  $k=k(\un x,\un y):= |\un x\wedge \un y|$.
Furthermore, let us decompose the postfixes $\tilde{ \un x}, \tilde{\un y}$  into longest possible blocks of digits of the same type:
\begin{align}
	\label{block}
	\tilde {\un x }=: \un b_1 \un b_2\dots \un b_r, \ \tilde {\un y}= \un c_1 \un c_2\dots \un c_q,
\end{align}
with
\[ \{1,2\}\ni\text{typ}(\un b_i)\neq \text{typ}(\un b_{i+1})\in \{1,2\},  \mbox{ and } \{1,2\}\ni\text{typ}(\un c_j)\neq \text{typ}(\un c_{j+1})\in \{1,2\}.\]
We denote the number of blocks in $\tilde{\un x},\tilde{\un y}$ by $r$ and $q$, respectively.
From the definition  of the edge set of $E(\mathrm{HM}_n)$, it follows that for any path $P(\un x,\un y)= (\un x = \un q^0, \dots, \un q^\ell = \un y)$, the consecutive vertices  on the path only differ in their postfixes, and these have different types. That is, each consecutive pair of vertices can be written in the form
\[ \forall i, \un q^i= \un w^i \un z^i, \ \un q^{i+1}= \un w^i \tilde{\un z}^{i}, \text{ with } \text{typ}(\un z^i)\ne \text{typ}(\tilde{ \un z}^{i})\in \{1,2\}.
\]
Now we fix an arbitrary self-map $p$ of $\Sigma$ such that
\[(x, p(x) ) \in E(G) \ \forall x \in G. \]
Most commonly, $p(p(x))\ne x$.
Note that $x$ and $p(x)$ have different types since $G$ is bipartite. For a word $\un z=(z_1 \dots z_m)$ with $\text{typ}(\un z )\in \{1,2\}$ we define $p(\un z):=(p(z_1) \dots p(z_m))$. Then, Def.~\ref{edgeset_def} implies that
\be \label{pairedge}
(\un t\un z, \un t p(\un z)) \text{ is an edge in } G_{\ell+m}, \forall \un t=(t_1\dots t_\ell).
\ee
Using \eqref{pairedge}, we construct a path $P(\un x, \un y)$ between two arbitrary vertices $\un x$ and $\un y$ that has length at most $r+q +\mathrm{diam}(G)-2$.
Starting from $\un x$ the first half of the path $P(\un x, \un y)$ is as follows:
\[
\ba 
\hat {\un x}^0&= \un x = (\un x \wedge \un y) \un b_1 \dots \un b_{r-1} \un b_r\\
\hat {\un x}^1&= (\un x \wedge \un y) \un b_1 \dots \un b_{r-1}p(\un b_{r})\\
&\dots \\
\hat {\un x}^{r-1}&= (\un x \wedge \un y) \un b_1 p(\un b_2 \dots p(\un b_{r-1}p(\un b_{r}))),\\
\ea
\]
Starting from $\un y$ the first half of the path $P(\un x, \un y)$ is as follows:
\[
\ba \label{yhat}
\hat {\un y}^0&= \un y = (\un x \wedge \un y) \un c_1 \un c_2\dots  \un c_r\\
\hat {\un y}^1&= (\un x \wedge \un y) \un c_1 \dots \un c_{r-1}p(\un c_{r})\\
&\dots \\
\hat {\un y}^{q-1}&= (\un x \wedge \un y) \un c_1 p(\un c_2 \dots p(\un c_{r-1}p(\un c_{q}))).\\
\ea
\]
It follows from \eqref{pairedge} that $P_x:=(\hat {\un x}^0, \hat {\un x}^1, \dots,\hat {\un x}^{r-1})$ and $P_y:=(\hat {\un y}^{q-1}, \dotsm \hat {\un y}^1, \hat {\un y}^0)$
are two paths in $\mathrm{HM}_n$.
To construct $P(\un x, \un y)$, it remains to connect $\hat {\un x}^{r-1}$ and $\hat {\un y}^{q-1}$. Using \eqref{pairedge} this can be done with a path $P_c$  of length at most $\mathrm{diam}(G)$. Indeed, since the postfixes $\un c_1 p(\un c_2 \dots p(\un c_{r-1}p(\un c_{q})))$ and $\un b_1 p(\un b_2 \dots p(\un b_{r-1}p(\un b_{r})))$ both have a type, one can connect them in at most as many edges as the diameter of the base graph\footnote{One can do this coordinate-wise by using the edge-connection rule described in \textbf{(b)} after Def.~\ref{edgeset_def}: Suppose $\un z=z_1z_2 \dots z_k$ and $\un v=v_1v_2\dots v_k$ are two vertices that both have a type. Then for each coordinate pair  $z_i, v_i$ we choose the shortest path on the base graph $G$ that connects them, that we denote by $\mathcal P_i$ with length $m_i<\mathrm{diam}(G)$. Then $\mathrm{dist}(\un z, \un v)=\max_{i}{m_i}$, and the path can be realized so that each coordinate follows the path $\mathcal P_i$ independently. The shorter paths simply stay put at their final vertex ($z_i$) once they are finished.}.   

Clearly,
\[ 
\mathrm{Length}(P(\un x, \un y ))\le r+q+\mathrm{diam}(G)-2\le 2(n-1)+\mathrm{diam}(G).
\]

For the lower bound on the diameter of $\mathrm{HM}_n$, we show that we can find two vertices in $\mathrm{HM}_n$ of distance $2(n-1)+\mathrm{diam}(G)$.
Pick two vertices with $|\un x \wedge \un y|=0$, so $x_1\ne y_1$ so that the distance between $x_1$ and $y_1$ in $G$ is exactly $\mathrm{diam}(G)$, and set each blocks $b_i$ and $c_i$ of length $1$. Note that in each step on any path between two vertices, the number of blocks in \eqref{block} changes by at most one. Further, since $x_1 \neq y_1$ to connect $\un x$ to $\un y$, we have to reach two vertices that have a type.
Starting from $\un x$, to reach the first vertex $\un a=(x_1\dots )$ of this property, we need at least $n-1$ steps on any path $\tilde P$. Similarly, starting from $\un y$, we need at least $n-1$ steps to reach the first vertex $\un b=(y_1 \dots)$ where all the digits are of the same type. Since the distance of  $x_{1}$ and $y_{1}$ in $G$ is $\mathrm{diam(G)}$, and we can change the first digit of a vertex on a path only to a neighbor digit in $G$ in one step on any path, we need at least $\mathrm{diam}(G)$ edges to connect $\un a$ to $\un b$.
\end{proof}

\bibliographystyle{plain}
\bibliography{cikk}

\end{document}